\documentclass[a4paper]{amsart}
\usepackage[latin9]{inputenc}
\usepackage{a4}
\usepackage{amsfonts}
\usepackage{amssymb}
\usepackage{amsmath}
\usepackage{amsthm}
\usepackage{changepage}
\usepackage{nomencl}
\usepackage{geometry}
\usepackage{tikz}
\usepackage{verbatim}
\usetikzlibrary{matrix}
\usepackage{hyperref}
\usepackage[toc,page]{appendix}
\usepackage{mathrsfs}
\usepackage[square, numbers, comma]{natbib}

\usepackage{enumerate}
\usepackage{amsbsy}

\usepackage[all]{xy}
\usepackage{pb-diagram,pb-xy}
   \bibpunct{[}{]}{,}{n}{}{,}
\input cyracc.def

\begin{document}
\providecommand{\keywords}[1]{\textbf{\textit{Keywords: }} #1}
\newtheorem{theorem}{Theorem}[section]
\newtheorem{lemma}[theorem]{Lemma}
\newtheorem{proposition}[theorem]{Proposition}
\newtheorem{corollary}[theorem]{Corollary}
\theoremstyle{definition}
\newtheorem{definition}{Definition}[section]
\theoremstyle{remark}
\newtheorem{remark}{Remark}
\newtheorem{conjecture}{Conjecture}[section]
\newtheorem{question}{Question}[section]
\newtheorem{example}{Example}[section]

\def\p{\mathfrak{p}}
\def\q{\mathfrak{q}}
\def\s{\mathfrak{S}}
\def\Gal{\mathrm{Gal}}
\def\Ker{\mathrm{Ker}}
\def\Coker{\mathrm{Coker}}
\newcommand{\cc}{{\mathbb{C}}}   
\newcommand{\ff}{{\mathbb{F}}}  
\newcommand{\nn}{{\mathbb{N}}}   
\newcommand{\qq}{{\mathbb{Q}}}  
\newcommand{\rr}{{\mathbb{R}}}   
\newcommand{\zz}{{\mathbb{Z}}}

\title{On fake subfields of number fields}
\author{Joachim K\"onig}
\address{Korea National University of Education, Department of Mathematics Education, 28173 Cheongju, South Korea}
\begin{abstract}
We investigate the failure of a local-global principle with regard to ``containment of number fields"; i.e., we are interested in pairs of number fields $(K_1,K_2)$ such that $K_2$ is not a subfield of any algebraic conjugate $K_1^\sigma$ of $K_1$, but the splitting type of any single rational prime $p$ unramified in $K_1$ and in $K_2$ is such that it cannot rule out the containment $K_2\subseteq K_1^\sigma$. Examples of such situations arise naturally, but not exclusively, via the well-studied concept of arithmetically equivalent number fields. We give some systematic constructions yielding ``fake subfields" (in the above sense) which are not induced by arithmetic equivalence. This may also be interpreted as a failure of a certain local-global principle related to zeta functions of number fields.
\end{abstract}
\maketitle
\section{Introduction and main results}

Two number fields $K_1$ and $K_2$ are called {\it arithmetically equivalent} if they have they same zeta function. This is equivalent to the property that all primes of $\mathbb{Q}$ which are unramified in $K_1$ and $K_2$ have the same splitting pattern in those two fields, i.e., their Frobenius has the same cycle type in the two induced permutation actions. In other words, $K_1$ and $K_2$ are in some sense indistinguishable by purely local investigation. 
Such number fields have featured prominently in the context of several arithmetic problems, such as the Davenport-Lewis-Schinzel problem on reducibility of variable separated equations, and the investigation of pairs of Kronecker conjugate polynomials, most notably by Fried (e.g, \cite{Fried}) and M\"uller (\cite{Mue}). 
See, e.g., \cite{Perlis} for an introduction as well as some systematic nontrivial examples of arithmetically equivalent number fields. Recently, stronger versions of 
arithmetic equivalence (e.g., \cite{Prasad}, \cite{Sutherland}) have also been studied, and links to other kinds of ``equivalence notions" in field arithmetic were explored e.g. in \cite{N12} and \cite{LN}. 
In this paper, we consider a natural generalization of this notion: two fields $K_1$ and $K_2$ for which purely local investigations (namely, of the splitting types of unramified primes) are insufficient to rule out the possibility that $K_2$ is a subfield of $K_1$.  Concretely, we make the following definitions.
\begin{definition}
Let $K_1$ and $K_2$ be number fields of degrees $n:=[K_1:\mathbb{Q}]$ and $m:=[K_2:\mathbb{Q}]$, such that $m$ divides $n$. We say that $K_2$ is \textit{locally sub-}$K_1$ if the following holds for all primes $p$ of $\mathbb{Q}$ which are unramified in $K_1$ and $K_2$: if $c_1,\dots, c_d$ are the residue degrees of the primes extending $p$ in $K_2$, then there exist positive integers $e_{i,1}, \dots, e_{i,r_i}$ such that $\sum_{j=1}^{r_i} e_{i,j} = \frac{n}{m}$ for all $i\in \{1,\dots, d\}$ and the residue degrees of primes extending $p$ in $K_1$ are exactly the $c_i\cdot e_{i,j}$ ($i=1,\dots, d$, $j=1,\dots, r_i$).
 When $K_2$ is {locally sub-}$K_1$, but not contained in any conjugate $K_1^\sigma$ of $K_1$ ($\sigma\in \textrm{Gal}(\overline{\mathbb{Q}}/\mathbb{Q})$), we shall also call $K_2$ a \textit{fake subfield} of $K_1$.
\end{definition}

Obviously, if $K_2\subseteq K_1$, then $K_2$ is locally sub-$K_1$ by the multiplicativity of residue degrees in towers of extensions. 
The name ``fake subfield" is inspired by a recent article of Corvaja \cite{Corv} on ``fake values" of polynomials, dealing with a failure of local-global principles similar in spirit to the one considered here, but with regard to value sets of polynomials.

The notion of $K_2$ being locally sub-$K_1$ is furthermore indeed a natural generalization of arithmetic equivalence, the latter being exactly the case of two fields being mutually locally subfields of each other; in other words, our notion yields an order relation on the set equivalence classes of number fields modulo arithmetic equivalence. 
Since any field $K_2$ arising from $K_1$ via arbitrary iteration of ``taking subfields" and ``taking arithmetically equivalent fields" will automatically be {locally sub-}$K_1$, a natural question is whether all examples of fake subfields arise from arithmetic equivalence. The answer will quickly be seen to be ``no". After reviewing the situation in small degrees in Section \ref{sec:small}, the main purpose of this article is to provide explicit infinite families of examples. In particular, we show the following:

\begin{theorem}
\label{thm:main}
For all odd primes $p$, the following hold:
\begin{itemize}
\item[a)]
Every degree $p+2$ number field $K$ with symmetric Galois group $S_{p+2}$ occurs as the fake subfield of some number field $F$.
\item[b)] There exist infinitely many solvable number fields of degree $2p$ occurring as the fake subfield of some number field $F$.
\end{itemize}
Furthermore, all these examples may be chosen such that they are not induced by arithmetic equivalence.
\end{theorem}

We will prove the two parts of Theorem \ref{thm:main} separately as Theorem \ref{thm:symm} and Theorem \ref{thm:solv}.
In Section \ref{sec:primitive}, we consider fake subfields of fields having no nontrivial subfields. For most of our results, translation of our main notions into permutation group-theoretical properties are crucial. These are contained, together with some other basic observations, in Section \ref{sec:prep}.
Towards the end of the paper, we will consider a strenghtened version of our main definition (Section \ref{sec:ram}), whose treatment requires some additional number theoretical ideas. We concludes with some open problems in Section \ref{sec:open}.

A few arguments require a moderate amount of computer calculations. These have been performed using Magma \cite{Magma}, and the corresponding code is included in an ancillary file.

{\bf Acknowledgements}: I thank Pietro Corvaja, as well as the two anonymous referees, for helpful suggestions. I am also indebted to Daniele Garzoni for pointing out Lemma \ref{lem:2trans}.

 

\section{Some preparations}
\label{sec:prep}
The following famous result, due to Gassmann (\cite{Gassmann}), turns the problem of arithmetic equivalence into a purely group-theoretical problem.
\begin{proposition}
\label{prop:gassmann}
Two number fields $K_1$ and $K_2$ are arithmetically equivalent if and only if the following hold: 1) the Galois closure of $K_1/\mathbb{Q}$ and of $K_2/\mathbb{Q}$ is the same, say $\Omega$; and 2) the subgroups $U:=\textrm{Gal}(\Omega/K_1)$ and $V:=\textrm{Gal}(\Omega/K_2)$ of $G:=\textrm{Gal}(\Omega/\mathbb{Q})$ fulfill that each conjugacy class of $G$ intersects $U$ and $V$ in the same number of elements, i.e., for all $g\in G$ one has $|g^G\cap U| = |g^G\cap V|$. (In this case, $U$ and $V$ are also said to be Gassmann-equivalent in $G$.)
\end{proposition}

\begin{corollary}
\label{cor:gassmann}
Let $K_1$ and $K_2$ be arithmetically equivalent number fields, with joint Galois closure $\Omega\supseteq\mathbb{Q}$. Let $G:=\textrm{Gal}(\Omega/\mathbb{Q})$, and consider  {\it any} transitive permutation action of $G$. Then $U_1:=\textrm{Gal}(\Omega/K_1)$ and $U_2:=\textrm{Gal}(\Omega/K_2)$ have the same number of orbits in this action.
\end{corollary}
\begin{proof}
By Gassmann's criterion (Proposition \ref{prop:gassmann}), $U_1$ and $U_2$ in particular contain the same amount of elements having exactly $d$ fixed points in the given action, for each $d\ge 0$. By the Cauchy-Frobenius formula, this means that $U_1$ and $U_2$ have the same number of orbits.
\end{proof}

We now turn to some first observations around our main notion of fake subfields. This notion has several implications on the relation of the two fields involved; notably, it is already part of the definition that the degree of a number field must be divisible by the degrees of all its fake subfields. The following lemma gives a further noteworthy relation.
\begin{lemma}
\label{lem:firstobs}
If $K_2$ is {locally sub-}$K_1$, then 
the Galois closure of $K_2/\mathbb{Q}$ is contained in the one of $K_1/\mathbb{Q}$. 
In particular, any number field can have only finitely many fake subfields.
\end{lemma}
\begin{proof}
Let $p$ be a prime which is totally split in the Galois closure of $K_1/\mathbb{Q}$. 
Since $K_2$ is locally sub-$K_1$, it follows that $p$ must also be totally split in the Galois closure of $K_2/\mathbb{Q}$. 
A well-known theorem  by Bauer (\cite{Bauer}) then asserts that the latter Galois closure is contained in the former one.
\end{proof}

For the following, some extra terminology will be useful: given a permutation $\sigma \in S_n$ consisting of exactly $d$ disjoint cycles of length $c_1\ge\dots \ge c_d$ (for short: cycle type $\lambda:=(c_1,\dots, c_d)$), and a total of $d$ cycle types $\mu_1 = (e_{1,1},\dots, e_{1,r_1}), \dots, \mu_d = (e_{d,1}, \dots, e_{d, r_d})$ of permutations in $S_d$, define the concatenation of $\lambda$ and $(\mu_1,\dots, \mu_d)$ to be the cycle type (in $S_{mn}$) of the form $(c_1\cdot e_{1,1}, \dots, c_1\cdot e_{1,r_1}, \dots, c_d\cdot e_{d,1}, \dots, c_d\cdot e_{d,r_d})$.  This corresponds exactly to the following situation in the context of splitting of prime ideals in number fields: If $K_1\supseteq K_2\supseteq \mathbb{Q}$ are number fields with $m=[K_2:\mathbb{Q}]$ and $n=[K_1:K_2]$, such that the rational prime $p$ decomposes into $d$ primes $\mathfrak{p}_1,\dots,  \mathfrak{p}_d$ of degrees $c_1,\dots, c_d$ in $K_2$, and each $\mathfrak{p}_i$ decomposes into $r_i$ primes of degree $e_{i,1},\dots, e_{i,r_i}$ in $K_1$, then the cycle type of Frobenius at $p$ in $K_1$ is exactly the concatenation defined above. In particular, a field $K$ being locally sub-$F$ is equivalent to the following: for each prime $p$ unramified in $F/\mathbb{Q}$, let $\lambda_1$ and $\lambda_2$, respectively, be the cycle type of Frobenius at $p$ in the Galois group of (the Galois closure of) $K/\mathbb{Q}$ and $F/\mathbb{Q}$ respectively; then there exists a tuple $\underline{\mu}:=(\mu_1,\dots,\mu_d)$ of cycle types such that the concatenation of $\lambda_1$ with $\underline{\mu}$ equals $\lambda_2$.

Using this, we can reword our main notions group-theoretically, namely in terms of cycle structures in the Galois groups.
\begin{lemma}
\label{lem:wreath}
Let $K_1$ and $K_2$ be number fields of degrees $n:=[K_1:\mathbb{Q}]$ and $m:=[K_2:\mathbb{Q}]$, such that $m$ divides $n$ and the Galois closure of $K_2$ is contained in the one of $K_1$. 
For $i\in \{1,2\}$, let $G_i$ be the Galois group of the Galois closure of $K_i/\mathbb{Q}$ (viewed in its induced degree-$[K_i:\mathbb{Q}]$ action), and let $\pi: G_1\to G_2$ be the restriction to the Galois closure of $K_2$. 
 Then $K_2$ is locally sub-$K_1$ if and only if the following holds: 
 For all $g\in G_1$, there exists an element $\sigma$ in the imprimitive wreath product $S_{k}\wr G_2\le S_{n}$ (with $k:=n/m$) whose cycle structure is the same as that of $g$, whereas the cycle structure of its projection onto $G_2$ (via action of the wreath product on blocks) is the same as that of $\pi(g)$.
\end{lemma}
\begin{proof}
Note that the cycle structures in the wreath product $S_k\wr G_2$ are exactly the concatenations of $\lambda$ and $(\mu_1,\dots, \mu_d)$ where $\lambda$ is a cycle type (consisting of $d\ge 1$ cycles) in $G_2$ and $\mu_1,\dots, \mu_d$ are cycle types in $S_k$. Thus the implication ``$\Leftarrow$" is immediate from the preceding. The implication ``$\Rightarrow$" follows in the same way, up to noting that, by the Frobenius density theorem (a weaker version of the Chebotarev density theorem), every cycle type in the Galois group occurs as the cycle type of Frobenius at infinitely many primes.
\end{proof}

One reason for the relevance of the notion of arithmetic equivalence is that arithmetically equivalent fields have the same zeta function. Below, we translate the more general notion of being ``locally sub-$K_1$" into a property of zeta functions. Recall that the (Dedekind) zeta function of a number field $K$ has an Euler product $\zeta_K(s) = \prod_{\mathfrak{p}} \frac{1}{1-{N_{K/\mathbb{Q}}(\mathfrak{p})^{-s}}}$, with the product being over all prime ideals $\mathfrak{p}$ of $O_K$. For a finite set $\mathcal{S}$ of rational primes, by the {\it contribution at $\mathcal{S}$ to $\zeta_K$}, we mean the product of all terms corresponding to primes extending some prime in $\mathcal{S}$. Of course, the contribution at unramified rational primes is completely determined by the cycle structure of their Frobenius.

\begin{lemma}
\label{lem:zeta}
The following are equivalent for number fields $K_1, K_2$:
\begin{itemize}
\item[1)] $K_2$ is {locally sub-}$K_1$.
\item[2)] 
For every finite set $\mathcal{S}$ of prime numbers unramified in $K_1$ and in $K_2$, 
there exists an extension $F:=F_{\mathcal{S}}$ of $K_2$ such that the contributions at $\mathcal{S}$ to the zeta functions of $F$ and of $K_1$ are the same.
\end{itemize}
\end{lemma}
\begin{proof}
2)$\Rightarrow$ 1) is obvious. 
For the converse, one may use the following: since the symmetric groups have generic Galois extensions over all number fields, it is possible (e.g., as a special case of a result by Saltman \cite[Theorem 5.9]{Saltman}), given any finite collection of primes $\mathfrak{p}$ of $K_2$ and for each such $\mathfrak{p}$ a cycle type in $S_n$ with $n:=\frac{[K_1:\mathbb{Q}]}{[K_2:\mathbb{Q}]}$, to find an $S_n$-extension $F\supseteq K_2$ having Frobenius class given by the prescribed cycle type at all those primes $\mathfrak{p}$. Concretely, choose as the finite set of primes $\mathfrak{p}$ of $K_2$ exactly those extending the given rational primes $p\in \mathcal{S}$ in $K_2$, and choose the cycle types such that, for each such $p\in \mathcal{S}$, the  cycle type of the Frobenius at $p$ in $K_1/\mathbb{Q}$ and in $F/\mathbb{Q}$ is the same (this is possible via suitable concatenation of cycle types, by Lemma \ref{lem:wreath}). Then $F$ achieves the claim of 2) for the given finite set of prime numbers.
\end{proof}

In general, it is indeed necessary to exclude the ramified primes from the characterizing condition 2) above; for this, we refer to Example \ref{ex:nonstrong}

Of course, if in Lemma \ref{lem:zeta}, the fields $F$ did {\it not} depend on the chosen set $\mathcal{S}$ of primes, one would indeed have arithmetic equivalence between $F$ and $K_1$. It is therefore natural to search for examples of fake subfields that do not arise from arithmetic equivalence; in the context of Lemma \ref{lem:zeta}, this may then be seen as the failure of a local-global principle (namely, for the zeta functions). 
The following is an elementary, but useful criterion to find or exclude candidates for this.

\begin{lemma}
\label{lem:intersect}
Assume that $K$ is a fake subfield of some number field $F$. Let $L$ be the Galois closure of $K/\mathbb{Q}$ and $G=\textrm{Gal}(L/\mathbb{Q})$ (viewed in its induced degree $[K:\mathbb{Q}]$ action). Finally, let $U:=\textrm{Gal}(L/L\cap F)\le G$. 
 Then the following hold. 
\begin{itemize}
\item[1)] $U$ does not fix a point, but
\item[2)] every element of $U$ has at least one fixed point.
\end{itemize}
\end{lemma}
\begin{proof}
Let $\Omega$ be the Galois closure of $F/\mathbb{Q}$, $\Gamma:=\textrm{Gal}(\Omega/\mathbb{Q})$ and $V:=\textrm{Gal}(\Omega/F)$. Due to Lemma \ref{lem:firstobs}) we have $L\subseteq \Omega$, and hence $U$ equals the restriction of $V$ to $L$. Since every $\sigma\in V$ has a fixed point in the degree-$[F:\mathbb{Q}]$ action of $\Gamma$, the fact that $K$ is {locally sub-}$F$ forces $\sigma$ to also fix a conjugate of $K$, i.e., every element of $U\le G$ has a fixed point. On the other hand, $U$ itself cannot fix a point, or otherwise, up to algebraic conjugates, the fixed field of $U$, and hence a fortiori $F$, would contain $K$, contradicting the notion of ``fake subfield". \end{proof}

Note that permutation groups $U\le S_n$ possessing no fixed point, but in which every element has a fixed points, have been considered in the context of ``intersective polynomials", i.e., polynomials having a root in (almost) every $\mathbb{Q}_p$, but not in $\mathbb{Q}$ itself. For essentially the same reason, they feature prominently in the context of ``fake values" of morphisms studied in \cite{Corv}.

\begin{corollary}
If $K/\mathbb{Q}$ is a Galois extension, then $K$ is not a fake subfield of any number field.
\end{corollary}
\begin{proof}
This follows directly from the previous lemma, upon noting that no non-identity element in the regular permutation action of a group has a fixed point.
\end{proof}

\section{Extensions of small degree}
\label{sec:small}

We begin our investigation of concrete examples of fake subfields by considering number fields of small degree. 

\begin{corollary}
If $[K:\mathbb{Q}]\le 4$, then $K$ cannot be a fake subfield of any number field.
\end{corollary}
\begin{proof}
One verifies directly that no transitive group of degree $\le 4$ has a fixed point free subgroup in which every element fixes a point (this also reflects the well-known fact that there are no non-trivially intersective polynomials of degree less than $5$). The assertion thus follows from Lemma \ref{lem:intersect}. 
\end{proof}

\subsection{Quintic fields}
For quintic fields $K/\mathbb{Q}$, the question of whether $K$ is a fake subfield of some number field is completely resolved by the solvability or non-solvability of the Galois group. 
\begin{theorem}
\label{thm:quintic}
\begin{itemize}
\item[a)]
Every $S_5$ and every $A_5$-quintic field occurs as the fake subfield of some number field, in a way not induced by arithmetic equivalence.
\item[b)] No solvable quintic field occurs as the fake subfield of a number field.
\end{itemize}
\end{theorem}


\begin{proof}
Since b) can again be derived straightaway by inspection using Lemma \ref{lem:intersect} (or alternatively, from Lemma \ref{lem:solv_p}), it suffices to prove a).

First, let $K$ be an $S_5$ quintic number field, $\Omega$ the Galois closure of $K/\mathbb{Q}$, and let $U:=\langle (1,2,3), (2,3)(4,5)\rangle (\cong S_3) \le S_5$.
We claim that a) $K$ is a fake subfield of the fixed field $\Omega^U$ of $U$, and b) this relation is not induced by arithmetic equivalence.
Indeed, the pairs of cycle structures of non-identity elements of $S_5$ in the two coset actions are exactly $((2.1^3), (2^{10}))$, $((2^2.1), (2^8.1^4))$, $((3.1^2), (3^6.1^2))$, $((4.1), (4^4.2^2))$, $((5), (5^4))$ and $((3.2), (6^3.2))$, all of which are compatible with $K$ being a fake subfield of $\Omega^U$.

Furthermore, the overgroups of $U$ of order dividing $24 = |\textrm{Gal}(\Omega/K)|$ are exactly $U$ itself and $S_3\times S_2$ (the full 2-set stabilizer).
None of these have Gassmann equivalent subgroups inside $S_5$; for example, the only other subgroup (up to conjugacy) of order $12$ other than $S_3\times S_2$ is $A_4$, which cannot be equivalent to the $2$-set stabilizer, since it contains no transposition; and the only other subgroups of order $6$ other than $U$ are $\langle (1,2,3), (1,2)\rangle\cong S_3$ and $\langle (1,2,3)(4,5)\rangle \cong C_6$; neither can be equivalent to $U$, since they contain no double transposition. In other words, there is no way to ``jump" from $U$ to $\textrm{Gal}(\Omega/K) \cong S_4$ via containment and arithmetic equivalence.

Next, let $K$ be an $A_5$-quintic field. Now the analog of the above construction with $U=\langle (1,2,3), (2,3)(4,5)\rangle \le A_5$ does not quite work; notably, in the (degree-$10$) action of $A_5$ on cosets of $U$, the 3-cycle has cycle structure $(3^3.1)$, which is not compatible with $(3.1^2)$. Instead, the following succeeds. 

Let $\Gamma = A_5\times C_3$, and consider the order $6$ subgroup $U=\langle (1,2,3), (2,3)(4,5)\rangle \le A_5$ (i.e., $[\Gamma:U]=30$). This is clearly not contained in an index-$5$ subgroup $V\cong A_4\times C_3$ of $\Gamma$.
Since elements of cycle structure $(3.1^2)$ did indeed constitute the {\it only} obstruction to the fixed field of $V$ being a fake subfield of the fixed field of $U\times C_3$, it suffices to verify that all preimages of such elements in $\Gamma$ no longer constitute an obstruction to the fixed field of $V$ being a fake subfield of the fixed field of $U$.  Those preimages either have cycle structure $(3^9.1^3)$ (namely, elements of order $3$ inside $A_5$, and hence contained in some conjugate of $U$) or $(3^{10})$ (namely,  elements of the form $(\sigma, \tau)$, where $\sigma\in A_5$ and $\tau \in C_3$ are each of order $3$; these are not contained in any conjugate of $U$). Clearly, there is no obstruction arising from either of these cycle structures together with the ``small" cycle structure $(3.1^2)$.
Also, once again, it is easy to verify that the ``fake subfield" relation is not induced by arithmetic equivalence (in fact, the only non-trivially Gassmann equivalent subgroups inside $\Gamma$ are pairs of subgroups of order $12$ both projecting to a point stabilizer $A_4$ inside $A_5$, i.e., arithmetic equivalence cannot be used to jump from an overfield of a degree-$5$ field to a field containing no such subfield).
\end{proof}

\subsection{Sextic fields}

\begin{theorem}
\label{thm:sextic}
\begin{itemize}
\item[a)] Every sextic number field whose Galois closure has Galois group $S_6$ or $A_6$ is a fake subfield of some number field, but this relation is always induced by arithmetic equivalence.
\item[b)]
Every sextic number field whose Galois closure has Galois group isomorphic to $A_4$  occurs as a fake subfield of some degree-$12$ number field, in a way not induced by arithmetic equivalence. 
\end{itemize}\end{theorem}
\begin{proof}
Regarding a), one can verify that the only subgroups $U$ of $S_6$ fulfilling the conditions of Lemma \ref{lem:intersect} are of order $4$, and up to conjugacy equal to $\langle (1,2)(3,4), (3,4)(5,6)\rangle$ (i.e., containing three double involutions, and having three orbits of length $2$). But then, if $F$ is a field having a sextic $S_6$- or $A_6$ number field $K$ as a fake subfield, as in Lemma \ref{lem:intersect}, $F$ has to contain the fixed field of $U$. This, however, is arithmetically equivalent to the fixed field of $\langle (1,2)(3,4), (1,3)(2,4)\rangle$ (by Gassmann's criterion, since this group also contains three double transpositions), but since the latter group fixes two points, its fixed field contains some conjugate of $K$.

Regarding b), denote by $V_4\le S_4$ the Klein $4$-group acting transitively on $4$ points, and by $a,b,c\in V_4$ the double transpositions. Set $G:= \{((x_1,x_2,x_3), y)\in V_4\wr C_3 = (V_4)^3\rtimes C_3\mid x_i\in V_4; y\in C_3; x_1x_2x_3=1\}$, acting as a transitive subgroup of the wreath product $V_4\wr C_3\le S_{12}$. (In Magma's database of transitive groups, cf.\ \cite{Hulpke}, $G$ is the group \texttt{TransitiveGroup}$(12,32)$.) Let $U_1\le G$ be a point stabilizer in this degree-$12$ action, i.e., without loss of generality $U_1=\{((x,x,1)\in G\cap (V_4)^3 \mid x\in V_4\}$. Let $N=\{(1,1,1), (a,b,c), (b,c,a), (c,a,b)\} \subset G\cap (V_4)^3$ be an order-$4$ normal subgroup of $G$ containing three fixed point free involutions. Then $G/N\cong A_4$, and we let $U_2\le G$ be a preimage of an order-$2$ subgroup of this quotient $A_4$ (i.e., $[G:U_2]=6$).  One verifies quickly that the only pairs of cycle structures of non-identity elements of $G$ in the action on cosets of $U_2$ and $U_1$ respectively are $((1^6), (2^6))$, $((2^2.1^2), (2^4.1^4))$, $((2^2.1^2), (2^6))$ and $((3^2), (3^4))$. From this, it is evident that, in any Galois extension of $\mathbb{Q}$ with group $G$, the fixed field of $U_2$ is locally a subfield of the fixed field of $U_1$, and even a fake subfield since $U_2$ contains no conjugate of $U_1$; indeed, otherwise $N$ would have to intersect a stabilizer in the degree-$12$ action in at least a subgroup of order $2$, which is clearly not the case by its definition via fixed point free involutions. Furthermore, this ``fake subfield" relation is not induced by arithmetic equivalence, notably because $U_1\cong V_4$ injects into $G/N$ and thus has three orbits (of length $2$) in the degree-$6$ action, whereas $U_2$ maps to an order-$2$ subgroup of $G/N$ and hence has four orbits,  whence Corollary \ref{cor:gassmann} is applicable. We have therefore obtained that the assertion of b) holds for every sextic $A_4$ number field which embeds into a $G$-extension of $\mathbb{Q}$. That this holds indeed for {\it every} $A_4$ number field follows from classical results on embedding problems (e.g., \cite[Chapter I, Theorem 2.4]{MM}), since $G$ is a semidirect product of $A_4$ and an abelian normal subgroup $C_2\times C_2$. This concludes the proof.
%
\end{proof}

From the observations so far, one also immediately obtains:
\begin{corollary}
The smallest degree $[K:\mathbb{Q}]$ of a number field $K$ possessing a fake subfield not induced by arithmetic equivalence is $12$.
\end{corollary}

\section{Extensions with symmetric Galois group}
We now progress to more systematic examples of fake subfields. In view of our definition of fake subfields, the following notion is useful:
Let $K_1$ and $K_2$ be two number fields, $\Omega$ the compositum of the Galois closures of $K_1$ and of $K_2$, and $\sigma\in G:=\textrm{Gal}(\Omega/\mathbb{Q})$. Let $\lambda_i$ be the cycle structure of $\sigma$ in the action induced by $K_i/\mathbb{Q}$ ($i=1,2$). If there exists a tuple $\underline{\mu}$ of cycle types such that the concatenation of $\lambda_2$ with $\underline{\mu}$ equals $\lambda_1$, then we say that $\sigma$ does not pose an obstruction to $K_2$ being locally sub-$K_1$ (resp., to $K_2$ being a fake subfield of $K_1$, if additionally no conjugate of $K_2$ is contained  in $K_1$). If $\sigma$ does not pose an obstruction to either $K_1$ being locally sub-$K_2$ and $K_2$ being locally sub-$K_1$ (in which case, of course, the cycle structures of $\sigma$ in the two actions must be the same), then we say that $\sigma$ does not pose an obstruction to $K_1$ and $K_2$ being arithmetically equivalent.

The following elementary observation was already known to Gassmann (and is, in fact, crucial to his criterion).
\begin{lemma}
\label{lem:nolocalobs}
Assume $K_0$, $K_1$, $K_2$ are number fields with $K_2$ contained in $K_0$, but not in any conjugate of $K_1$.
Let $\Omega/\mathbb{Q}$ be the compositum of the Galois closures of $K_0$ and of  $K_1$ over $\mathbb{Q}$, and let $G:=\textrm{Gal}(\Omega/\mathbb{Q})$ and $U_i:=\textrm{Gal}(\Omega/K_i)$ for $i=0,1,2$. Assume that $\sigma\in G$ is such that the following holds: for all $k\in \mathbb{N}$, the number of fixed points of $\sigma^k$ in the action on cosets of $U_0$ and of $U_1$ is the same. Then $\sigma$ does not pose an obstruction to $K_0$ and $K_1$ being arithmetically equivalent. In particular, $\sigma$ does not pose an obstruction to $K_2$ being a fake subfield of $K_1$.
\end{lemma}

We now extend the above results about quintic $S_5$ number fields to infinite families of Galois groups.
\begin{theorem}
\label{thm:symm}
Let $p\ge 3$ be a prime number and $K$ a (degree-$p+2$) $S_{p+2}$-number field. Then $K$ is a fake subfield of some number field, in a way not induced by arithmetic equivalence. In particular, there are infinitely many such number fields $K$ for each $p$.
\end{theorem}

\begin{proof}
We begin by noting that existence of infinitely many $S_n$-extensions of $\mathbb{Q}$ for each $n\in \mathbb{N}$ was already shown by Hilbert in \cite{Hilbert}.

Now choose subgroups $U_0,U_1,U_2\le S_{p+2}$ as follows: $U_2$ is the stabilizer of a point (in the natural action of $S_{p+2}$), $U_0\cong D_p$ is a dihedral group of order $2p$ contained in the two-point stabilizer of $S_{p+2}$ (in particular, up to conjugates, contained in $U_2$), and $U_1\cong D_p$ is a dihedral group with orbit lengths $p$ and $2$ (in particular, not being contained in $U_2$, even up to conjugates). This means that all non-identity elements of $U_0$ are $p$-cycles or involutions with three fixed points, and all non-identity elements of $U_1$ are $p$-cycles or involutions with one fixed point.
Letting $\Omega/\mathbb{Q}$ be any Galois extension with Galois group $S_{p+2}$, we will verify that there exists no element $\sigma\in S_{p+2}$ posing an obstruction to $K_2$ being a fake subfield of $K_1$, where $K_i$ denotes the fixed field of $U_i$ for $i=0,1,2$. We distinguish the following cases:
\begin{itemize}
\item[i)] $\sigma$ powers to a $p$-cycle. In this case $\sigma$ is necessarily either itself a $p$-cycle, or has cycles of length $p$ and $2$. Since $U_0$ and $U_1$ are of the same order containing the same number of $p$-cycles, every $p$-cycle is contained in the same number of conjugates of $U_0$ and of $U_1$, meaning that $p$-cycles have the same number of fixed points in the two coset actions. Also, elements of cycle type $(p.2)$ are contained in no conjugate of $U_0$ and $U_1$, thus have no fixed point in either action. By Lemma \ref{lem:nolocalobs}, the elements $\sigma$ considered here pose no local obstruction to $K_2$ being a fake subfield of $K_1$.
\item[ii)] $\sigma$ powers to an involution with exactly one fixed point. We claim that, in the action on cosets of $U_1$, $\sigma$ has only cycles of length $d$ and $d/2$, where $d:=ord(\sigma)$. Indeed, $\sigma^{d/2}$ and $\sigma^d=id$ are necessarily the only powers of $\sigma$ contained in some conjugate of $U_1$, readily yielding the claim. We next determine the proportion of $d/2$-cycles of $\sigma$ in this coset action. 


A well-known formula gives the number of fixed points of $\tau\in G$ in the action on cosets of $U$ as $\frac{[G:U]\cdot |\tau^G\cap U|}{|\tau^G|}$. Evaluating with $G=S_{p+2}$, $U=U_1\cong D_p$ as above, and $\tau$ an involution with one fixed point, yields $2^{a-1} \cdot a!$ with $a:=\frac{p+1}{2}$. This is bounded from above by $\frac{1}{p+2}[G:U] = \frac{(p+1)!}{2p}$, for all $p\ge 3$. Since the number of these fixed points is exactly the number of elements contained in $d/2$-cycles of $\sigma$, we have obtained that at most a proportion of $\frac{1}{p+2}$ of all elements are contained in such ``short" cycles. 
We now form a cycle type in the symmetric group on $|U_2|/|U_1| = \frac{(p+1)!}{2p}$ letters comprising as many $d/2$-cycles as contained in the above coset representation of $\sigma$, and only $d$-cycles otherwise; we have only verified that the proportion of $d/2$-cycles is small enough for this, and furthermore the cycle lengths thus formed do indeed add up to the required permutation degree $|U_2|/|U_1|$: indeed, not only the large permutation degree $[G:U_1]=\frac{(p+2)!}{2p}$, but also the quotient $\frac{(p+2)!/(2p)}{p+2} = \frac{(p+1)!}{2p}$ of the two permutation degrees is necessarily divisible by $d$; this follows readily from the fact that $\sigma$ is a permutation in $S_{p+1}$ of order coprime to $p$. 

Then, concatenating the fixed point of $\sigma$ with the thus constructed cycle type, as well as concatenating all other cycles of $\sigma$ (say, of length $m$) in the natural action by a partition consisting only of cycles of length $d/m$, yields exactly the cycle structure of $\sigma$ in the action on cosets of $U_1$, showing that there is no obstruction coming from $\sigma$.

\item[iii)] $\sigma$ powers to an involution with exactly three fixed points. In this case, since we may already assume $p>3$ due to Theorem \ref{thm:quintic}, no non-identity power of $\sigma$ fixes a point in the action on cosets of $U_1$, i.e., all cycle lengths of $\sigma$ in this action are identical (hence, equal to $d:=ord(\sigma)$). Clearly, $\sigma$ then does not pose an obstruction either (just concatenate 
any given cycle of length $e|d$ in the natural action by an element of $S_{|U_2|/|U_1|}$ with suitably many cycles of length $d/e$. This is possible since $d$ necessarily divides the quotient of the two permutation degrees, as in Case ii).)

\item[iv)] $\sigma$ is none of the above. In this case, no non-identity power of $\sigma$ fixes any point in the action on cosets of either $U_0$ or $U_1$, meaning that Lemma \ref{lem:nolocalobs} can be applied again.
\end{itemize}

Finally, assume that the above relation is induced by (containment of fields and) arithmetic equivalence. By Corollary \ref{cor:gassmann}, together with the fact that $U_2$ and $U_1$ have orbit lengths $(p+1,1)$ and $(p,2)$ respectively, this can only happen if $U_1$ has some intransitive overgroup $V_1$ (i.e., still with orbits of length $p$ and $2$) which is Gassmann-equivalent to a subgroup $V_2$ with orbit lengths $(p+1,1)$. Gassmann equivalence means in particular that every element of $V_1$ has a fixed point (since every element of $V_2$ does). Clearly, $V_1$ acts faithfully on its orbit of length $p$ (otherwise it would contain a $p$-cycle and a transposition with disjoint support, the product of which would yield a fixed point free element). In this situation, Theorem 3.3 of \cite{EKT} yields the existence of a normal subgroup $N\triangleleft V_1$ of index $2$ such that every element in $V_1\setminus N$ has exactly one fixed point on the length $p$ orbit. In particular, every $2$-point stabilizer of $V_1$ on this orbit is contained in $N$, which in particular means that $N$ cannot act $2$-transitively on this orbit. So $N$ is a transitive, but not $2$-transitive group of prime degree $p$. By a classical result due to Burnside (\cite{Burnside}), $N$ -- and hence also $V_1$ -- is then solvable, i.e., more concretely, is isomorphic to a subgroup of  $AGL_1(p)\cong C_p\rtimes C_{p-1}$. On the other hand, $V_2$ acts transitively on its orbit of length $p+1$ and contains a $p$-cycle, i.e., is $2$-transitive. Hence $|V_2|\ge p(p+1)$, whereas $|V_1|\le p(p-1)$, contradicting their Gassmann equivalence. This completes the proof.
\end{proof}

\begin{remark}
The degree restriction in Theorem \ref{thm:symm} should not be expected to be necessary at all for the conclusion to hold; in fact, it may reasonable to conjecture the same result for all $S_n$-extensions of sufficiently large degree $n$, although a proof might be combinatorially and group-theoretically intricate.
\end{remark}

\section{Solvable number fields}
We now consider the phenomenon of fake subfields among solvable number fields.

\begin{lemma}
\label{lem:solv_p}
If the Galois group of the Galois closure of $K/\mathbb{Q}$ acts as a Frobenius group,\footnote{Frobenius groups are often, but not always solvable.} then $K$ is not a fake subfield of any number field. In particular, no solvable number field of prime degree occurs as a fake subfield of any number field.
\end{lemma}
\begin{proof}
Let $G$ be a Frobenius group with Frobenius kernel $N$, and assume that a subgroup $U\le G$ as in Lemma \ref{lem:intersect} exists. Since all elements of $N\setminus\{1\}$ are fixed point free, $U\cap N=\{1\}$, and $UN$ is again a Frobenius group. But it is well-known that in such a group, all complements to $N$ are conjugate to each other, meaning that $U$ is in fact a point stabilizer of $UN\le G$, contradicting its definition.
The second assertion is immediate from the first, since solvable groups of prime degree are cyclic or Frobenius groups.
\end{proof}

\begin{theorem}
\label{thm:solv}
 Let $p\ge 5$ be a prime, and let $G=C_2\wr C_p\le S_{2p}$ be the imprimitive wreath product of cyclic groups of order $2$ and $p$. Then every degree-$2p$ number field with Galois group $G$ occurs as the fake subfield of some number field, in a way not induced by arithmetic equivalence.
\end{theorem}
\begin{proof}
Upon relabelling the elements of $\{1,\dots, 2p\}$ suitably, $G$ is generated by the transposition $(1,2)$ together with the double-$p$-cycle $(1,3,5,\dots, 2p-1)(2,4,6,\dots, 2p)$. Consider now the subgroup $U$ of $G$ generated by the two involutions $(1,2)(3,4)\dots(p,p+1)$ and $(p,p+1)(p+2,p+3)\dots (2p-1,2p)$. We will show that the fixed field $K_2$ of a point stabilizer in $G\le S_{2p}$ is a fake subfield of the fixed field $K_1$ of $U$ (of course, there are then infinitely many such fields, since the solvable group $G$ is well-known to occur as the Galois group of infinitely many Galois extensions of $\mathbb{Q}$). To find out about the cycle structures in the action of $G$ on cosets on $U$, note first that the only cycle structures in $G(\le S_{2p})$ are those of powers of the $2p$-cycle, as well as involutions with any number of transpositions between $1$ and $p$. Elements $x\in G$ such that no non-identity power of $x$ is contained in a conjugate of $U$ clearly consist only of cycles of length $\textrm{ord}(x)$ in that coset action, and then it is obvious that such an element does not pose an obstruction to $K_2$ being a fake subfield of $K_1$. On the other hand, the only elements which do have nontrivial powers inside a conjugate of $U$ (and, in fact, are themselves contained), are involutions with $\frac{p+1}{2}$ transpositions (namely two of them inside $U$, and clearly conjugate to each other in $G$ via a suitable power of the double-$p$-cycle) and with $p-1$ transpositions (namely, one of them inside $U$). We show that none of these pose an obstruction either; since $[K_1:\mathbb{Q}]/[K_2:\mathbb{Q}]$ is a $2$-power, it is sufficient to show that the proportion of fixed points of these elements in the action on cosets of $U$ is no more than the proportion of fixed points in the degree-$2p$ action (as indeed, one can then pass from the cycle type in the ``small" action to the one in the ``large" one, simply by concatenating each fixed point of the given involution $x\in G$ with an involution with the right amount of fixed points).

Using again the expression $\frac{[G:U]\cdot |x^G\cap U|}{|x^G|}$ for the fixed point number of $x$ in the action on cosets of $U$, the fixed point {\it proportion} obviously becomes $\frac{|x^G\cap U|}{|x^G|}$. In our case, we are reduced to $|x^G\cap U| \in \{1,2\}$, with the case $|x^G\cap U|=1$ obviously not posing an obstruction to the required proportion; in the other case (namely, $x$ an involution with $\frac{p+1}{2}$ transpositions), it suffices to note that the point stabilizer in the degree-$2p$ action also contains at least two conjugates of $x$; to see this, note that this point stabilizer is, up to conjugacy, equal to $\langle(3,4), (5,6),\dots, (2p-1,2p)\rangle$, which has $\begin{pmatrix} p-1 \\ (p+1)/2\end{pmatrix}$ involutions of the required cycle type; this is $\ge p-1$ since $p>3$. 

Moreover, due to Corollary \ref{cor:gassmann}, if the above relation were induced by (containment and) arithmetic equivalence, then $U$ would necessarily have at least as many orbits as the point stabilizer in the degree-$2p$ action. However, $U$ has $p$ orbits (all of length $2$), whereas the point stabilizer has two fixed points and orbit lengths $2$ otherwise, i.e., has $p+1$ orbits.
\end{proof}

Theorem \ref{thm:main}b) is now an immediate consequence of Theorem \ref{thm:solv} together with Theorem \ref{thm:sextic}b) (covering the case $p=3$) and Shafarevich's theorem asserting the existence of infinitely many Galois extensions of $\mathbb{Q}$ with any prescribed solvable Galois group.

\section{Fake subfields of fields with primitive Galois group}
\label{sec:primitive}
We so far mainly focussed on the question whether certain fields can {\it occur} as fake subfields of other fields. In this section, we take a moment to consider the opposite viewpoint, i.e., we ask whether certain fields can {\it have} fake subfields. 
A particularly interesting case seems to be the one where a field $K_1$ has no non-trivial subfields $\mathbb{Q}\subsetneq F\subsetneq K_1$. In terms of the Galois group, the latter is equivalent to saying that $U:=\textrm{Gal}(\Omega/K_1)$ is a maximal subgroup of $G:=\textrm{Gal}(\Omega/\mathbb{Q})$, where $\Omega$ denotes the Galois closure of $K_1/\mathbb{Q}$; equivalently, $G$ acts primitively on cosets of $U$. It is well-known that such a field $K_1$ can nevertheless have arithmetically equivalent fields not conjugate to $K_1$ (notably, for $d\ge 3$, the group $PGL_d(q)$ has two Gassmann equivalent classes of maximal subgroups of index $\frac{q^d-1}{q-1}$, corresponding to the stabilizer of a line and a hyperplane in $GL_d(q)$). Below we give an example where $K_1$ has no nontrivial arithmetically equivalent fields, but nevertheless possesses fake subfields; in other words, while $K_1/\mathbb{Q}$ has no nontrivial intermediate fields, purely local observations would suggest the existence of such subfields.
\begin{lemma}
\label{lem:234}
There exist infinitely many number fields of degree $234$ possessing no nontrivial subfields and no nontrivial arithmetically equivalent fields, but possessing fake subfields of degree $13$. More precisely, every Galois extension of $\mathbb{Q}$ with Galois group isomorphic to $PSL_3(3)$ contains such fields.
\end{lemma}
\begin{proof}
Let $G=PSL_3(3)$. This group is known to occur as the Galois group of infinitely many Galois extensions of $\mathbb{Q}$, cf.\ \cite{Malle}. Furthermore, $G$ possesses maximal subgroups $U\cong S_4 (\cong PGL_2(3))$, and there is no subgroup of $G$ Gassmann-equivalent to $U$ other than the conjugates of $U$ (e.g., the only other class of order-$24$ subgroups of $PSL_3(3)$ has elements of order $6$, which $S_4$ does not). Nevertheless, comparing the natural (degree $13$) permutation action of $PSL_3(3)$ with the degree-$234 (=13\cdot 18)$ action on cosets of $U$ yields that the fixed fields of the index-$13$ subgroups are fake subfields of the fixed field of $U$ (and since the latter fields have neither nontrivial subfields nor nontrivially arithmetically equivalent fields, this relation cannot be induced by arithmetic equivalence). Indeed, computation with Magma yields that the pairs of cycle structures of non-identity elements of $G$ in both actions are as follows: $((2^4.1^5), (2^{108}. 1^{18}))$, $((3^3. 1^4), (3^{78}))$, $((3^4. 1), (3^{77}. 1^3))$, $((4^2. 2^2. 1), (4^{54}. 2^8. 1^2))$, $((6.3.2.1^2), (6^{36}. 3^6))$, $((8.4.1), (8^{27}.4^4.2))$ and $((13), (13^{18}))$; now it is an easy exercise to compose the ``small" cycle structures with suitable cycle structures in $S_{18}$ to obtain all the ``large" cycle structures.
\end{proof}

\begin{remark} 
\begin{itemize}
\item[a)] An exhaustive search of the Magma database of primitive groups confirms that the examples in Lemma \ref{lem:234} are indeed smallest possible, both with regard to the degree  of the ``larger" and of the ``smaller" field involved. For such computations, it is useful to note that, under the given assumptions, both fields must have the same Galois closure. Indeed, if the Galois closure of the ``small" field were strictly smaller, the point stabilizer in the large degree action, having no nontrivial overgroups, would have to surject onto the whole group under projection to the smaller Galois group, which is incompatible with Lemma \ref{lem:intersect}. 
\item[b)] It would be interesting to exhibit not only an infinite family of fields, but an infinite family of groups leading to examples as in Lemma \ref{lem:234}. They seem to be rare among primitive groups, although I have computationally verified $PSL_3(p)$ to yield examples for $p=3,5,7$. Verifying whether this generalizes to all odd primes $p$ might be feasible, although beyond the scope of this article.
\end{itemize}
\end{remark}

We conclude this section by noting that examples such as the above can no longer occur if the assumption of primitivity of the Galois group is strengthened to $2$-transitivity.
\begin{lemma}
\label{lem:2trans}
Let $K$ be a number field, denote by $\Omega$ the Galois closure of $K/\mathbb{Q}$ and assume that $\textrm{Gal}(\Omega/\mathbb{Q})$ acts as a doubly transitive group. Then $K$ has no fake subfields of degree strictly less than $[K:\mathbb{Q}]$.
\end{lemma}
\begin{proof}
Let $G:=\textrm{Gal}(\Omega/\mathbb{Q})$ and $H< G$ any subgroup of index smaller than $[K:\mathbb{Q}]$. We claim that, since $G$ is doubly transitive, $H$ must necessarily be transitive. The assertion then follows easily, since the point stabilizer $U:=\textrm{Gal}(\Omega/K)$ in $G$ is then conversely transitive in the action on cosets of $H$ and so possesses a fixed point free element in this action, whence the fixed field of $H$ cannot be a fake subfield of $K$ by Lemma \ref{lem:intersect}. The claim itself seems to be reasonably well known (e.g., Exercise 2.12.3 in \cite{Cameron}), but we give a proof for completeness. Let $\pi$ be the permutation character of the action of $G$ on conjugates of $K$. Since this action is doubly transitive, $\pi = 1_G + \chi$ for an irreducible character $\chi$ (with $1_G$ the principal character). The number of orbits of $H$ in this action is then $\langle (1_G+\chi)_{| H}, 1_H\rangle$, which by Frobenius reciprocity equals $\langle 1_G+\chi, \psi\rangle$ where $\psi$ is the permutation character of $G$ in the action on cosets of $H$. Here $\langle 1_G, \psi\rangle=1$, and since $\deg(\psi) < \deg(\pi)$, $\psi-1_G$ must be a sum of characters of degree $<\deg\chi$. Hence, the scalar product equals $1$, i.e., $H$ is transitive, completing the proof.
\end{proof}

\section{Taking into account the ramified primes}
\label{sec:ram}
While the exclusion of ramified primes in our main definitions is convenient due to the arising reductions to group theory, it is of course natural (in particular, in the context of local-global principles and their failure) to strengthen the definitions to include also the ramified primes. The natural way to do this is as follows: 

\begin{definition} Say that $K_2$ is {\it strongly locally sub-$K_1$}, if for each rational prime $p$, the following holds: if $K_{{2, \mathfrak{p}_1}}$, $\dots$, $K_{2, \mathfrak{p}_r}$ are the completions of $K_2$ at all the primes $\mathfrak{p}_i$ extending $p$, and $K_{{1, \mathfrak{q}_1}}$, $\dots$, $K_{1, \mathfrak{q}_s}$ are the completions of $K_1$ at all the primes $\mathfrak{q}_j$ extending $p$, then there exists a partition of $\{1,\dots, s\}$ into $r$ subsets $M_i$ ($i=1,\dots, r$) such that for all $j\in M_i$, the field $K_{1,\mathfrak{q}_j}$ contains (some conjugate of) $K_{2,\mathfrak{p}_i}$, and additionally $\sum_{j\in M_i}[K_{1,\mathfrak{q}_j} : K_{2,\mathfrak{p}_i}]$ equals the same value (namely, automatically $[K_1:\mathbb{Q}]/[K_2:\mathbb{Q}]$) for all $i=1,\dots, r$. 
\end{definition}

Clearly, when restricting to unramified primes, this definition becomes equal to our previous definition of being locally sub-$K_1$, due to unramified extensions of $\mathbb{Q}_p$ being uniquely identified by their degree. The analogous strengthening has also been considered in the setting of arithmetic equivalence, the strengthened notion being denoted as ``locally equivalent", e.g., in \cite{LMM}. In both cases, it is not difficult to find examples demonstrating that the ``weak" version in general does not imply the ``strong" one (and hence, systematically finding examples for the strong version may require not only investigation of the structure of the Galois group, but also solution of additional arithmetic problems).

We give one example (not induced by arithmetic equivalence) of number fields $K_1$ and $K_2$ such that $K_2$ is locally, but not strongly locally sub-$K_1$.
\begin{example}
\label{ex:nonstrong}
We use the notation of the proof of Theorem \ref{thm:sextic}b). Let $V_4=\{1,a,b,c\}\le S_4$ be the Klein 4-group, and let $G:=\{((x_1,x_2,x_3), y)\in V_4\wr C_3 = (V_4)^3\rtimes C_3\mid x_i\in V_4; y\in C_3; x_1x_2x_3=1\}$. Now, additionally define $U_3:=\langle(a,b,c), (a,a,1)\rangle \le G\cap (V_4)^3$. I.e., $U_3\cong C_2\times C_2$. 
Choose any odd prime $p$. Then it follows from very general existence results on Galois extensions with prescribed local behavior that there exists a Galois extension of $\mathbb{Q}$ with Galois group $G$, ramified at $p$ with inertia group generated by the fixed point free involution $\sigma:=(a,b,c)$ and with decomposition group conjugate to $U_3$. Indeed, since $G$ is a semidirect product of two elementary-abelian groups of coprime order, it possesses a generic Galois extension over $\mathbb{Q}$ by \cite[Theorem 3.5]{Saltman}, which by Theorem 5.9 of the same paper implies realizability of $G$ as a Galois group with prescribed local behavior at any finitely many given primes. Recall now from the proof of Theorem \ref{thm:sextic}b) that any such Galois extension gives rise to sextic number fields $K_2$ being fake subfields of certain degree-$12$ number fields $K_1$, and we now analyze the residue degrees of primes extending $p$ in these fields $K_1$ and $K_2$. These are simply the number of orbits of $\langle\sigma \rangle$ joined to one common orbit of $U_3$ in the respective actions of degree $12$ and $6$. In the degree-$12$ action, only one pair of transpositions of $\sigma$ (namely, coming from the component entry ``$b$") is joined together, meaning that in $K_1$, $p$ is extended by one prime of residue degree $2$ and primes of degree $1$ otherwise. On the other hand,  $\sigma$ is in the kernel of the degree-$6$ action while $U_3$ maps to a cyclic group generated by a double transposition. This means that two pairs of fixed points of $\sigma$ are joined to two transpositions, i.e., in $K_2$, $p$ is extended by two primes of residue degree $2$ (and primes of degree $1$ otherwise). Therefore, $p$ poses an obstruction to $K_2$ being strongly locally sub-$K_1$; in fact, in view of Lemma \ref{lem:zeta}, it is worth noting that due to the above, no degree-$12$ number field containing $K_2$ can even have the same contribution to the zeta function at $p$ as $K_1$.
\end{example}

At least, we can show the following:
\begin{theorem}
\label{thm:loc_cyc}
Assume that $K_2$ is locally sub-$K_1$, and $K_1$ is ``locally cyclic", i.e., all primes ramifying in the Galois closure of $K_1/\mathbb{Q}$ have cyclic decomposition group. 
Then $K_2$ is strongly locally sub-$K_1$.
\end{theorem}
\begin{proof} Let $p$ be a prime ramifying in $K_1$, let $\sigma$ be a generator of the decomposition group at any prime extending $p$ in the Galois closure $\Omega$ of $K_1/\mathbb{Q}$, and let $\lambda_i$ be the cycle structure of $\sigma$ in the action on cosets of $\textrm{Gal}(\Omega/K_i)$ ($i=1,2$). Cycles of $\lambda_i$ are then in one-to-one correspondence with primes extending $p$ in $K_i$, with the cycle length equaling the degree over $\mathbb{Q}_p$ of the respective completion. Since $K_2$ is locally sub-$K_1$, $\lambda_1$ can be written as a concatenation of $\lambda_2$ with some $\underline{\mu} = (\mu_1,\dots, \mu_r)$ (where $r$ is the number of primes $\mathfrak{p}_i$ extending $p$ in $K_2$). In this way, each $\mathfrak{p}_i$ is assigned a finite set of primes $\mathfrak{q}_{i,j}$ of $K_1$ with the property that the degree $[(K_1)_{\mathfrak{q}_{i,j}}: \mathbb{Q}_p]$ of each completion is a multiple of $[(K_2)_{\mathfrak{p}_i}:\mathbb{Q}_p]$. But since the completion at any prime extending $p$ in all of $\Omega$ is a cyclic extension of $\mathbb{Q}_p$, subextensions are uniquely identified by their degree, meaning that divisibility of degrees automatically implies containment of fields. This proves the assertion.
\end{proof}

Certain heuristics in inverse Galois theory suggest that locally cyclic Galois extensions of $\mathbb{Q}$ should exist for every finite group, but this is of course a hard problem. Among the groups we have considered so far, we can conclude the following.
\begin{corollary}
There are infinitely many degree-$n$ $S_n$-number fields for $n\in \{5,7,9\}$, and infinitely many solvable number fields of degree $2p$ for every prime $p\ge 3$, which are fake subfields in the strong sense of some number field.
\end{corollary}
\begin{proof} As seen in the proof of Theorems \ref{thm:symm} and \ref{thm:solv}, these number fields all occur as fake subfields of some number field with Galois group $S_{p+2}$ ($p\in \{3,5,7\}$) or a solvable group, respectively. The assertion now follows from Theorem \ref{thm:loc_cyc}, together with the observation that all these groups are realizable as the Galois group of infinitely many locally cyclic extensions of $\mathbb{Q}$ (Theorems 5.1 and 5.4 in \cite{KK21}; in fact, the claim about solvable groups is already known due to Shafarevich's method for realizing solvable groups as Galois groups). \end{proof}

\section{Open questions}
\label{sec:open}
Due to Gassmann's criterion, the question whether a given number field $K$ has any non-trivially arithmetically equivalent fields translates to a pure question about the Galois group of its Galois closure over $\mathbb{Q}$. Is this also true for the question whether $K$ occurs as a fake subfield of some number field? This does not seem automatic, even though I do not know of any counterexamples.
Indeed, we have an elementary necessary group-theoretical condition for a number field $K$ to occur as the fake subfield of some number field $F$ (Lemma \ref{lem:intersect}), but it remains unclear whether this condition is also sufficient. We have at least seen (e.g., in the case $G=A_5$ in Theorem \ref{thm:quintic}, and in the case $G=A_4$ acting on six points, in Theorem \ref{thm:sextic}b)) that it may of course in general be necessary to look for fields $F$ outside of the Galois closure of $K/\mathbb{Q}$. 
Looking a bit closer at such examples, one sees that, in the case $G=A_5$, the obstruction for finding a fake overfield inside the same Galois closure was relatively ``harmless": it came from the cycle type $(3.1^2)$, which had cycle structure $(3^3.1)$ in the only possible (namely, degree-$10$) candidate action of $A_5$. The obstruction arises essentially from the two degrees being too close to each other - it vanishes once the cycles in the second cycle type are simply repeated sufficiently many times (for example, three times, leading to the cycle structure $(3^9.1^3)$, which is exactly what happens when passing from the above degree-$10$ action of $A_5$ to the degree-$30$ action of $A_5\times C_3$, as we did in the proof of Theorem \ref{thm:quintic}). Other cases require more intricate solutions. Notably, if there is an element $\sigma\in G$ whose fixed point proportion in the ``small" (i.e., candidate for a fake subfield) permutation action is smaller than in the action on cosets of $U$ (in the notation of Lemma \ref{lem:intersect}), then this will remain an obstruction even after arbitrarily repeating all the cycles from the second permutation action; in other words, passing to a direct product cannot get rid of the obstruction in such a case, nor can passing to a full wreath product. There are, of course, many other possibilities to embed $G$ as a quotient into an imprimitive group whose blocks action equals the action on cosets of $U$, and it may well be possible to give some inductive construction to successively get rid of obstructions, but the group theory should become rather delicate in the process.

On a related note, in the case where one has to go beyond the Galois closure of $K/\mathbb{Q}$ in search of fields containing $K$ as a fake subfield, one has to solve certain embedding problems, corresponding to group extensions $1\to N\to \Gamma \to G=\textrm{Gal}(K/\mathbb{Q})\to 1$. It seems worth wondering whether it could ever happen that all group extensions leading to fields containing $K$ as a fake subfield have to be {\it non-split} extensions of the Galois group of $K/\mathbb{Q}$. In such a case, the solvability of the embedding problem would in general depend on the behavior of primes in $K/\mathbb{Q}$, i.e., it could happen that, among such fields with the same Galois group, some occur as fake subfields whereas others don't. Whether this phenomenon can indeed occur, I do not know.


\begin{thebibliography}{9}
\bibitem{Bauer} M.\ Bauer, \textit{Zur Theorie der algebraischen Zahlk\"orper}. Math.\ Annalen 77 (1916), 353--356.
\bibitem{Magma} W.\ Bosma, J.\ Cannon, C.\ Playoust, \textit{The Magma algebra system. I. The user language}. J. Symb.\ Comput.\ 24 (1997), 235--265.
\bibitem{Burnside} W.\ Burnside, \textit{On the properties of groups of odd order}. Proc.\ London Math.\ Soc.\ 33 (1900). 162--185.
\bibitem{Cameron} P.J.\ Cameron, \textit{Permutation groups}. London Math.\ Soc.\ Student Texts 45, Cambridge Univ.\ Press, 1999.
\bibitem{Corv} P.\ Corvaja, \textit{On the local-to-global principle for value sets}. Riv.\ Mat.\ Univ.\ Parma 13 (2022), 47--72.
\bibitem{EKT} C.\ Elsholtz, B.\ Klahn, M.\ Technau, \textit{On polynomials with roots modulo almost all primes}. Acta Arithm.\ 205 (2022), 251--263.
\bibitem{FKS} B.\ Fein, W.\ M.\ Kantor, M.\ Schacher, \textit{Relative Brauer groups II}. J.\ Reine Angew.\ Math.\ 328 (1981), 39--57.
\bibitem{Fried} M.\ D.\ Fried, \textit{The field of definition of function fields and a problem in the reducibility of polynomials in two variables}. Illinois J.\ Math.\ 17 (1973), 128--146.
\bibitem{Gassmann} F.\ Gassmann, \textit{Bemerkungen zur vorstehenden Arbeit von Hurwitz (\"Uber Beziehungen zwischen den Primidealen eines algebraischen K\"orpers und den Substitutionen seiner Gruppe)}. Math.\ Z.\ 25 (1926), 665--675.
\bibitem{Hilbert} D.\ Hilbert, \textit{\"Uber die Irreducibilit\"at ganzer rationaler Functionen mit ganzzahligen Coefficienten}. J.\ Reine Angew.\ Math. 110 (1892), 104--129.
\bibitem{Hulpke} A.\ Hulpke, \textit{Constructing transitive permutation groups}.
 J.\ Symb.\ Comput.\ 39 (1) (2005), 1--30.
\bibitem{KK21} K.-S.\ Kim, J.\ K\"onig, \textit{On Galois extensions with prescribed decomposition groups}. J.\ Number Theory 220 (2021), 266--294.
\bibitem{LMM} B.\ Linowitz, D.\ B.\ MacReynolds, N.\ Miller, \textit{Locally equivalent correspondences}. Ann.\ Inst.\ Fourier 67(2) (2017), 451--482.
\bibitem{LN} A.\ Lubotzky, D.\ Neftin, \textit{Sylow-conjugate number fields}. Preprint (2021). \texttt{https://arxiv.org/pdf/2201.04103}.
\bibitem{Malle} G.\ Malle, \textit{Polynomials for primitive nonsolvable permutation groups of degree $d \le 15$}. J.\ Symb.\ Comput.\ 4 (1987), 83--92.
\bibitem{MM} G.\ Malle, B.H.\ Matzat, \textit{Inverse Galois Theory}. 2nd edition. Springer, Berlin, 2018.
\bibitem{Mue} P.\ M\"uller, \textit{Kronecker conjugacy of polynomials}. Trans.\ Amer.\ Math.\ Soc.\ 350 (5) (1998), 1823--1850.
\bibitem{N12} D.\ Neftin, \textit{Admissibility and field relations}. Isr.\ J.\ Math 191 (2012), 559--584.
\bibitem{Perlis} R.\ Perlis, \textit{On the equation $\zeta_k(s) = \zeta_{k'}(s)$}. J.\ Number Theory 9 (3) (1977), 342--360.
\bibitem{Prasad} D.\ Prasad, \textit{A refined notion of arithmetically equivalent number fields, and curves with isomorphic Jacobians}. Adv.\  Math.\ 312 (2017), 198--208.
\bibitem{Saltman} D.\ J.\ Saltman, \textit{Generic Galois extensions and problems in field theory}. Adv.\ Math.\ 43 (1982), 250--283.
\bibitem{Sutherland} A.\ V.\ Sutherland, \textit{Stronger arithmetic equivalence}. Discrete Analysis 2021:23 (2021), 23 pp.
\end{thebibliography}
\end{document}